\documentclass[letterpaper, 10 pt, conference]{ieeeconf}  
\usepackage{mathalfa}
\usepackage{color}
\usepackage{tikz}
\usepackage{booktabs}
\usepackage{xcolor}
\usepackage{units}
\usepackage[thinlines]{easytable}

\IEEEoverridecommandlockouts                              

\usepackage{graphicx} 
\usepackage{amsmath, amssymb}

\newtheorem{definition}{Definition}

\title{\LARGE \bf
Optimal Control of Nonholonomic Systems via Magnetic Fields*
}

\author{Maria Oprea$^{1}$, Max Ruth$^{1}$, Dora Kassabova$^{2}$, and William Clark$^{2}$
\thanks{*This work was funded by NSF grant DMS-1645643.
M. Oprea was supported by the Army Research Office Biomathematics Program Grant W911NF-18-1-0351.
M. Ruth was supported by the National Science Foundation Graduate Research Fellowship under Grant No. DGE-1650441.}
\thanks{$^{1}$M. Oprea and M. Ruth are with the Center for Applied Mathematics, Cornell University, Ithaca, NY 14850, USA {\tt\small mao237@cornell.edu, mer335@cornell.edu}}%
\thanks{$^{2}$D. Kassabova and W. Clark are with the Department of Mathematics, Cornell University, Ithaca, NY 14850, USA {\tt\small dmk285@cornell.edu, wac76@cornell.edu}}
}

\newcommand{\abs}[1]{\left| #1 \right|}

\newtheorem{proposition}{Proposition}
\newtheorem{corollary}{Corollary}

\begin{document}

\maketitle
\thispagestyle{empty}
\pagestyle{empty}

\begin{abstract}
Geometric optimal control utilizes tools from differential geometry to analyze the structure of a problem to determine the control and state trajectories to reach a desired outcome while minimizing some cost function. For a controlled mechanical system, the control usually manifests as an external force which, if conservative, can be added to the Hamiltonian. 
In this work, we focus on mechanical systems with controls added to the \textit{symplectic form} rather than the Hamiltonian. In practice, this translates to controlling the magnetic field for an electrically charged system. We develop a basic theory deriving necessary conditions for optimality of such a system subjected to nonholonomic constraints.
We consider the representative example of a magnetically charged Chaplygin Sleigh, whose resulting optimal control problem is completely integrable.
\end{abstract}

\section{INTRODUCTION}

Optimal control of mechanical systems has been extensively studied in the past \cite{pontryagin,https://doi.org/10.1002/zamm.19630431023,lee,Bloch1994ReductionOE}. Furthermore, it is commonly known that deforming the canonical symplectic form by incorporating a magnetic field results in a new symplectic form \cite{tony,marsden}. 
In physical application, magnetic fields are used in a variety of control problems; for example, magnetic fields in combination with electric quadrupoles are commonly used for confinement of particles in Penning traps, e.g.~for both quantum trapping \cite{major_charged_2005,perez-rios_how_2013,brif_control_2010,deffner_optimal_2014}, and for mass spectrometry \cite{martikyan_application_2021}. 
Another application is the manipulation and navigation of micro-robots in different fluid environments; magnetic fields are a popular external actuation tool as they allow fuel-free remote control and a high degree of programmability \cite{KOLEOSO2020100085,biomedical,10.3389/frobt.2022.834177,https://doi.org/10.1002/aisy.202000267,doi:10.1021/acs.chemrev.0c01234}. While newer technologies have led to an increase in magnetic experimental work, the theory on optimal control of magnetic systems is scarce and underdeveloped.

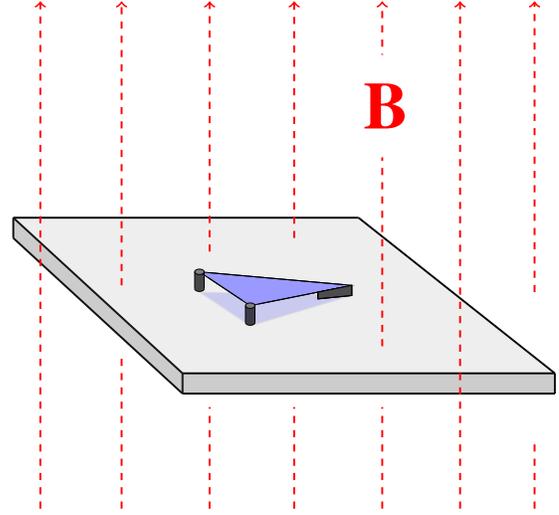
\begin{figure}
    \centering
    \scalebox{0.9}{
    	\begin{tikzpicture}\label{Dora's cool pic}
    		
    		
    		\draw[thick](0,0.3) -- (0,0);
    		\draw[thick](2.5,-2) -- (2.5,-2.3);
    		\draw[thick](8,-2) -- (8,-2.3);
    		
    		\draw [draw=gray, fill=black, opacity=0.2]
    		(2.5,-2.3) -- (8,-2.3) -- (8,-2) -- (2.5,-2) -- cycle;
    		
    		\draw[thick](0,0.3) -- (5.1,0.3);
    		\draw[thick](0,0.3) -- (2.5,-2);
    		\draw[thick](5.1,0.3) -- (8,-2);
    		\draw[thick](2.5,-2) -- (8,-2);
    		
    		\draw [draw=gray, fill=black, opacity=0.07]
    		(0,0.3) -- (5.1,0.3) -- (8,-2) -- (2.5,-2) -- cycle;
    		
    		\draw[thick](0,0) -- (2.5,-2.3);
    		\draw[thick](2.5,-2.3) -- (8,-2.3);
    		
    		\draw [draw=gray, fill=black, opacity=0.2]
    		(0,0.3) -- (2.5,-2) -- (2.5,-2.3) -- (0,0) -- cycle;
    		
    		
    		\draw[dashed,->, color = red,thick]  (0.4,-4) -- (0.4,3.5);
    		
    		
    		\draw[dashed, color = red,thick]  (7.7,-4) -- (7.7,-3);
    		\draw[dashed, color = red,thick]  (5.45,-4) -- (5.45,-2.5);
    		\draw[dashed,color = red,thick]  (4.15,-4) -- (4.15,-2.5);
    		\draw[dashed, color = red,thick]  (2.9,-4) -- (2.9,-2.5);
    		\draw[dashed,color = red,thick]  (1.6,-4) -- (1.6,-1.7);
    		
    		
    		\draw[dashed,->, color = red,thick]  (7.7,-0.8) -- (7.7,3.5);
    		\draw[dashed,->, color = red,thick]  (6.6,-4) -- (6.6,3.5);
    		\draw[dashed,->, color = red,thick]  (5.45,-1.6) -- (5.45,3.5);
    		\draw[dashed,->, color = red,thick]  (4.15,0.0) -- (4.15,3.5);
    		\draw[dashed,->, color = red,thick]  (2.9,-0.2) -- (2.9,3.5);
    		\draw[dashed,->, color = red,thick]  (1.6,-0.7) -- (1.6,3.5);

    		\draw [draw= white, fill=white, opacity=1]
    		(6.5,2.7) -- (6.5,1.2) -- (5,1.2) -- (5,2.7) -- cycle;
    		
    		\node[text = red, scale = 2.8] at (5.5,1.95) {\textbf{B}};

    		\node (r0) at ( 3.5,  -1) {}; 
    		\node (s0) at (2.75, -0.5) {}; 
    		\node (s1) at ( 5, -0.70) {}; 
    		
    		\draw [draw=gray, fill=blue, opacity=0.15]
    		(3.5,  -1.3) -- (2.75, -0.8) -- (5, -0.85) -- cycle;
    		
    		\draw [draw=gray, fill=white, opacity=1]
    		(3.5,  -1) -- (2.75, -0.5) -- (5, -0.7) -- cycle;
    		
    		\fill[fill=blue!40] (r0.center)--(s0.center)--(s1.center);
    		\path[draw] (r0)--(s0);
    		\path[draw] (s0)--(s1);
    		\path[draw] (s1)--(r0);
    		
    		\draw (3.5,  -1)--(2.75, -0.5);
    		\draw (2.75, -0.5)--( 5, -0.70);
    		\draw ( 5, -0.70)--(3.5,  -1);
    		
    		\draw [black,fill=black!50] ( 3.5,  -1) ellipse (0.07 and 0.04);
    		\draw (3.43,-1) -- (3.43,-1.25);
    		\draw (3.43,-1.25) arc (180:360:0.07 and 0.04);
    		\draw (3.57,-1.25) -- (3.57,-1);  
    		\fill [black,opacity=0.7] (3.43,0) -- (3.43,-1.25) arc (180:360:0.07 and 0.04) -- (3.57,-1) arc (0:180:0.07 and -0.04);
    		
    		\draw [black,fill=black!50] (2.75, -0.5) ellipse (0.07 and 0.04);
    		\draw (2.68,-0.5) -- (2.68,-0.75);
    		\draw (2.68,-0.75) arc (180:360:0.07 and 0.04);
    		\draw (2.82,-0.75) -- (2.82,-0.5);  
    		\fill [black,opacity=0.7] (2.68,-0.5) -- (2.68,-0.75) arc (180:360:0.07 and 0.04) -- (2.82,-0.5) arc (0:180:0.07 and -0.04);
    		
    		\draw [draw=black, fill=white, opacity=1]
    		(4.5,-0.9) -- (5,-0.85) -- (5,-0.7) -- (4.5,-0.80) -- cycle;
    		
    		\draw [draw=black, fill=black, opacity=0.7]
    		(4.5,-0.9) -- (5,-0.85) -- (5,-0.7) -- (4.5,-0.80) -- cycle;

    		\draw[thin](4.5,-0.9) -- (5,-0.85);
    		\draw[thin](4.5,-0.9) -- (4.5,-0.80);
    		\draw[thin](5,-0.85) -- (5,-0.7);
    \end{tikzpicture}}
    \caption{The Chaplygin sleigh subject to the left-invariant magnetic field, $\mathcal{B} = Bdx_c\wedge dy_c$.}
	\label{fig:Dora's cool pic}
\end{figure}

Typically, when studying optimal control for Hamiltonian systems, controls, denoted by $u$, are implemented via an external force accompanying the Hamiltonian \cite{pontryagin,https://doi.org/10.1002/zamm.19630431023,lee,Bloch1994ReductionOE}:

\begin{equation*}
    i_X\omega = dH + \pi_Q^*F(u),
\end{equation*}
which can be directly incorporated to the Hamiltonian if conservative, $\tilde{H}(u) = H + V(u)$.

In this paper we approach the optimal control problem through the lens of distorting the symplectic form denoted by $\omega$. Rather than manipulating the Hamiltonian, here we impose the controls by directly manipulation the symplectic form, i.e.
$\omega: \mathcal{U} \rightarrow \Omega^2(T^*Q)$ is a two form valued function on $\mathcal{U}$, the space of admissible controls. In practice, changes in $\omega$ can be achieved by acting on the system with a magnetic field which can be viewed as a closed two-form on the configuration space $Q$. 

In formulating our optimal control problem, we will consider systems with nonholonomic constraints which are assumed to be linear in the velocities. In particular we will illustrate this theory in the famous case of the Chaplygin sleigh in the plane.
Using the developed theory, we are able to determine the conditions necessary to reverse the velocity of the sleigh. It turns out that this problem is completely integrable independent on the choice of cost function. 




The equations of motion for the symplectically controlled nonholonomic system are 
\begin{equation*}
    i_X\omega_u = dH + \lambda^i\eta_i,
\end{equation*}
where $\eta^\alpha\in\Omega^1$ are the constraints and the multipliers, $\lambda_\alpha$, are chosen to enforce these constraints. The optimal control problem is to determine the control trajectory in $\mathcal{U}$ such that the following integral is minimized,
\begin{equation*}
    J = \min_u \, \int_0^T \, \ell(x, u) \, dt. 
\end{equation*}


The goal of this work is two-fold. First, we present the theory for control of magnetic mechanically controlled systems. This is accomplished by deriving the equations of motion and the conservation laws associated with these dynamics. Second, we apply optimal control theory, specifically Pontraygin's maximum principle, to these systems with specific attention drawn to the electrically charged Chaplygin sleigh subjected to an external magnetic field.

The dynamics of nonholonomic systems subject to a magnetic field are developed in \S\ref{sec:dynamics} along with the result that energy is always preserved in such systems. The optimal control problem is addressed in \S\ref{sec:control} where it is shown that the original energy of the system and the optimal control Hamiltonian always Poisson commute. These results are applied to the Chaplygin sleigh in \S\ref{sec:chap}. Numerical results from this example are presented in \S\ref{sec:results}. Conclusions and future work is in \S\ref{sec:conclusions}.


\section{MAGNETIC NONHOLONOMIC SYSTEMS} \label{sec:dynamics}


There has been research done on connecting optimal control theory and magnetic systems, \cite{tony}. However, to our knowledge there is a lack of treatment where the magnetic field is the control parameter.
We start by introducing the standard symplectic magnetic field representation, and then develop the resulting equations of motion. Finally, we show that the total energy is always conserved under the controlled motion.

Let $Q$ be an $n$ dimensional manifold and let $\pi_Q:T^*Q\to Q$ be the cotangent bundle projection.
Let $H:T^*Q\to \mathbb{R}$ be a natural Hamiltonian, i.e. it is of the form kinetic plus potential energy.
Let the system be subjected to $k<n$ linear nonholonomic constraints; this restricts the dynamics to the nonintegrable distribution $\mathcal{D}\subset TQ$. As these constraints are linear, we can (locally) find 1-forms $\eta_i$ for $i \in \{1, \dots, k\}$ such that the span of $\eta_i$ annihilates $\mathcal{D}$ i.e. for all vectors $v \in \mathcal{D}$, $\eta_i(v) = 0$.

Moreover, assume the system is under the influence of a magnetic field given by a closed 2-form $\mathcal{B} \in \Omega^2(Q)$. In coordinates, we can represent $\mathcal{B}$ as an antisymmetric $n\times n$ matrix with entries $B_{ij}$. Under these assumptions, the equations of motion are given by \cite{https://doi.org/10.48550/arxiv.1410.5682}, \cite{symmetry} \color{black}:
\begin{equation}\label{eq:eom_general_hamiltonian}
\begin{split}
     i_{X_\mathcal{B}}\omega_\mathcal{B} &= dH + \lambda^i\pi^*_Q\eta_i,\\
     \omega_\mathcal{B} &= \omega + \pi^*_Q \mathcal{B},
\end{split}
\end{equation}
where $\omega = dq^i\wedge dp_i$ is the canonical symplectic form.
Throughout this paper, Einstein summation convention will be utilized; repeated indices imply summation. 
Equation \eqref{eq:eom_general_hamiltonian} produces a vector field, $X_B = (\dot{q}, \dot{p})$, which lies tangent to the induced co-distribution $\mathcal{D}' = \mathbb{F}H^{-1}(\mathcal{D})\subset T^*Q$.

In the special case of left-invariant systems, magnetic Lie-Poisson reduction can be implemented, cf. Chapter 7 in \cite{marsden}. Let $G$ be a Lie group, $H:T^*G\to\mathbb{R}$ a left-invariant Hamiltonian, $\mathcal{D}\subset TG$ a left-invariant constraint distribution, and $\mathcal{B}\in \Omega^2(G)$ a left-invariant and closed 2-form. Then the nonholonomic magnetic Lie-Poisson equations are given by
\begin{equation}\label{eq:magnetic_Lie-Poisson}
    \dot{\mu} = \mathrm{ad}^*_{dh}\mu + \mathcal{B}_e(dh,\cdot) + \lambda^i\eta_i,
\end{equation}
where $h = H|_{\mathfrak{g}^*}$ is the restriction to the identity and $\eta_i\in\mathfrak{g}^*$ forms a basis of the annihilator to $\mathcal{D}_e\subset \mathfrak{g}$. Again, the multipliers are chosen to enforce the constraints.


For both the general magnetic nonholonomic equations of motion and their reduced counterparts, a modified 2-form was utilized by adding a closed and basic 2-form. Fortunately, this new form will always be symplectic \cite{marsden}. This closed and basic form will henceforth be called $\mathcal{B}$ and will be our control parameter. 
In many physical systems, this can be actualized by introducing a magnetic field. The special propriety of introducing control through a magnetic form is that, independent of the control, the original Hamiltonian is always preserved. 

\begin{proposition}\label{prop:energy_conservation_magnetic}
    For any $\mathcal{B} \in \Omega^2(Q)$, the Hamiltonian $H$ is preserved under the magnetic nonholonomic flow given by \eqref{eq:eom_general_hamiltonian} i.e. if $X_\mathcal{B}$ is given by \eqref{eq:eom_general_hamiltonian}, then $\mathcal{L}_{X_\mathcal{B}}H = 0$, where  $\mathcal{L}$ denotes the usual Lie derivative.
\end{proposition}
\begin{proof}
It follows from $\omega_\mathcal{B}$ being symplectic, and the fact that
$(\pi_Q)_*X_\mathcal{B} \in \mathrm{Ann}(\eta_i)$, for all $i \in \{1, \dots, k\}$.
\end{proof}



\section{OPTIMAL CONTROL PROBLEM}\label{sec:control}
We wish to determine optimal controls for the systems \eqref{eq:eom_general_hamiltonian} and  \eqref{eq:magnetic_Lie-Poisson}. 
Specifically, we want to solve a bounded horizon optimal control problem by minimizing the following cost functional:
\begin{equation}\label{eq:cost}
    J(B) = \int_0^T \ell(q, p,B)dt,
\end{equation}
subject to the fixed endpoints,
\begin{equation*}
    \begin{array}{cc}
    q(0) = q_0, & q(T) = q_f, \\
    p(0) = p_0, & p(T) = p_f,
    \end{array}
\end{equation*}
whose dynamics evolves according to either \eqref{eq:eom_general_hamiltonian} or  \eqref{eq:magnetic_Lie-Poisson} and  $\ell:\mathcal{D}'\times\mathcal{U}\to\mathbb{R}$ is the running cost. Recall that $B$ refers to the matrix entries of the 2-form $\mathcal{B}$ and thus one can think of $B\in \mathbb{R}^N$ for some $N$.

\begin{figure}
    \centering
    \includegraphics[scale = 0.45]{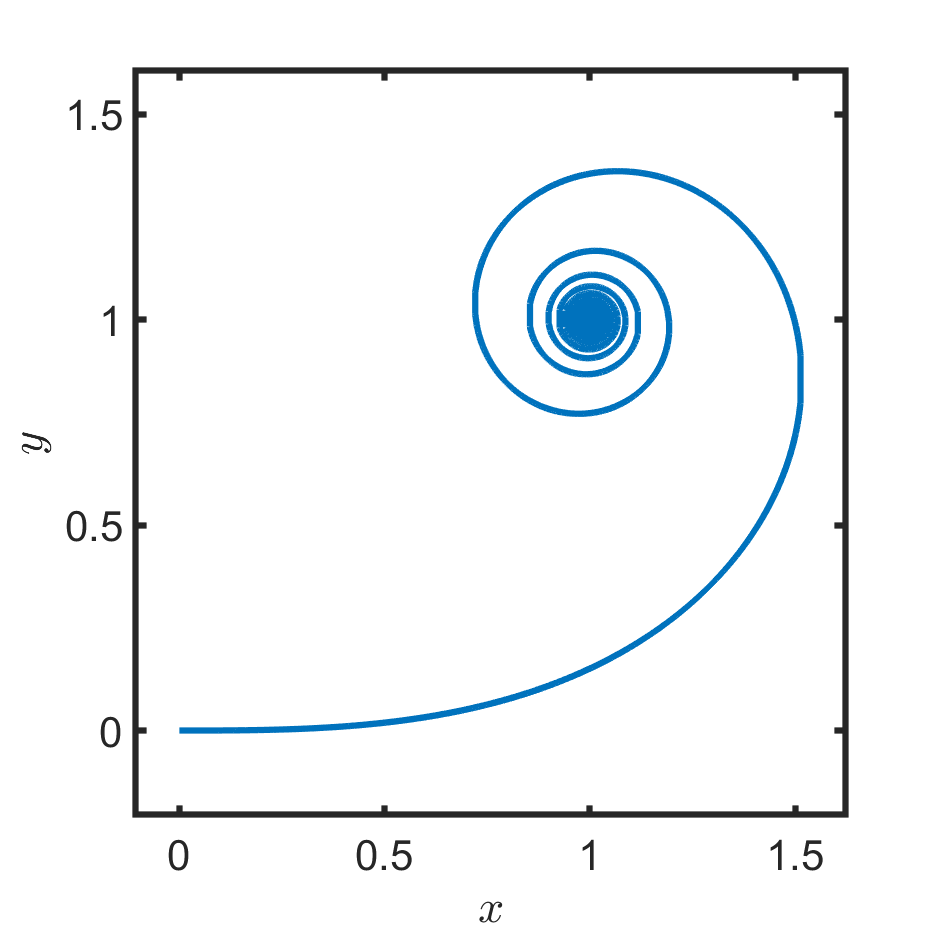}
    \includegraphics[scale = 0.45]{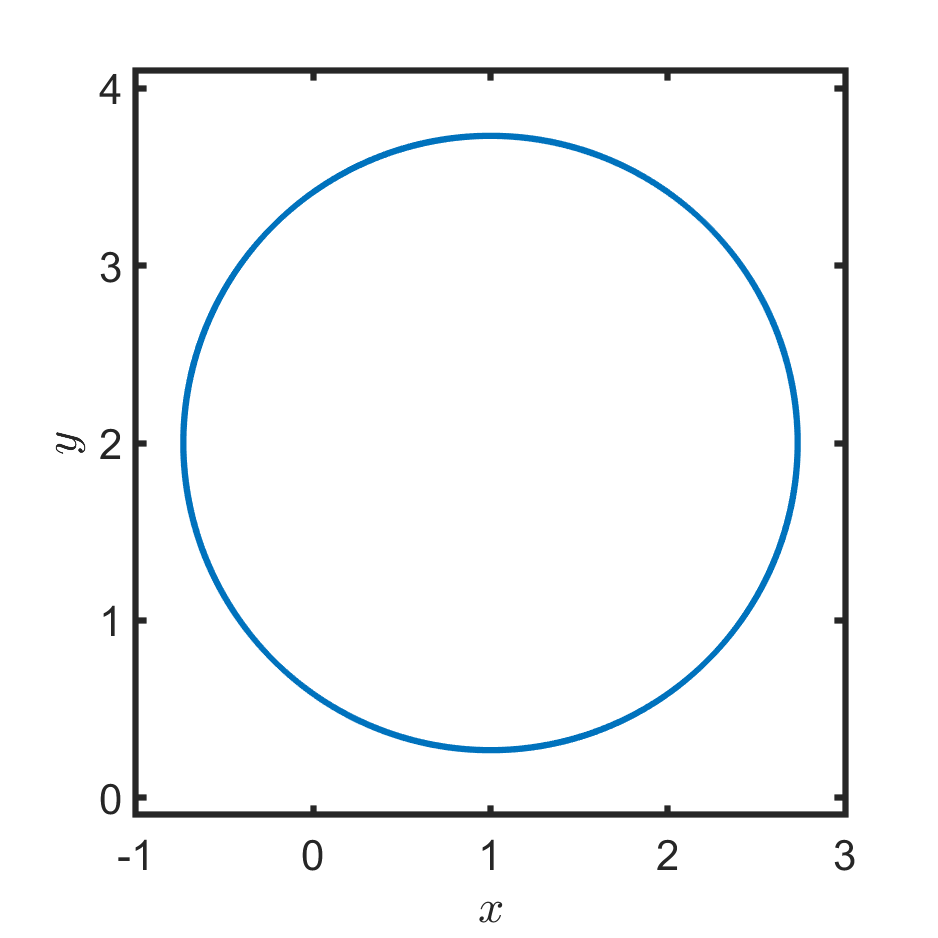}
    \includegraphics[scale = 0.45]{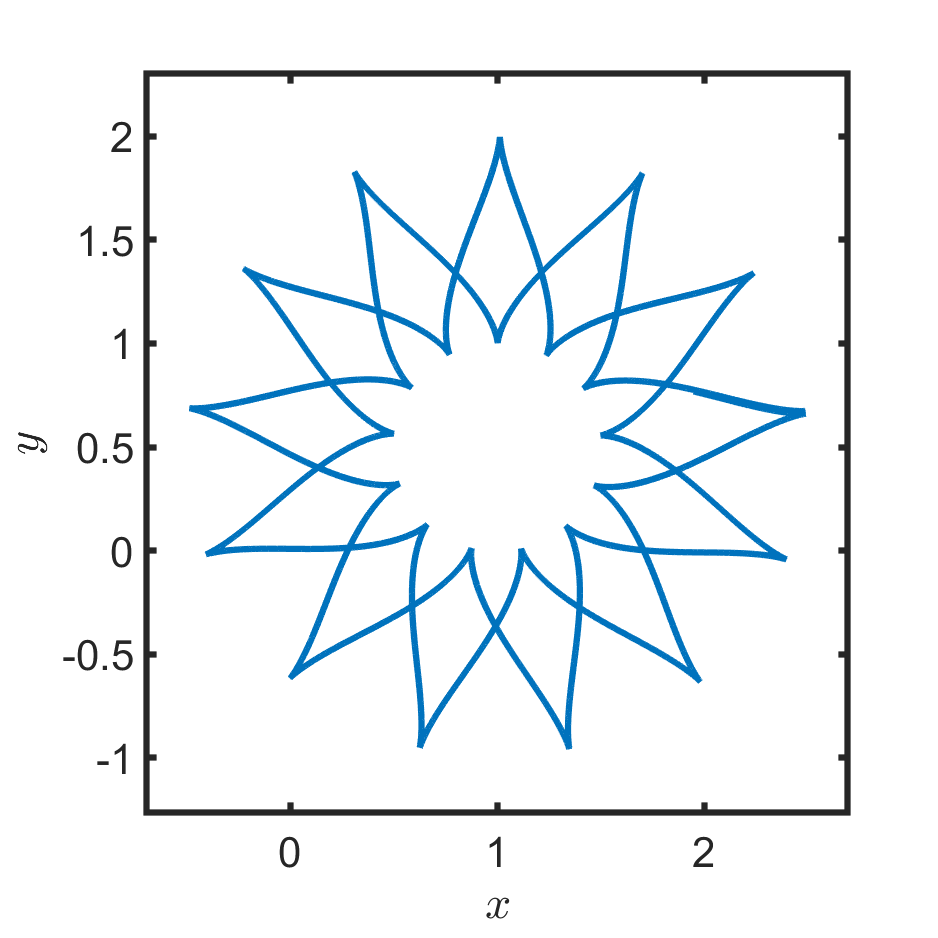}
    \caption{Three trajectories of the Chaplygin sleigh with constant $B$ for $B < 1$, $B=1$, and $B>1$.}
    \label{fig:spiro}
\end{figure}

To solve the optimal control problem we construct the extended Hamiltonian. 

\begin{definition}
For $(q,p)\in \mathcal{D}'$, let $(q,p,p_q,p_p)$ be the induced coordinates on $T^*\mathcal{D}'$.
Then the extended Hamiltonian on $T^*\mathcal{D}'$ is given by
    \begin{gather}\label{eq:extende_Hamiltonian}
        \tilde{H}:\mathcal{U}\times \mathcal{D}'\to\mathbb{R}\\
        (B, q, p, p_q, p_p) \mapsto \ell(q,p;{B}) + \langle (p_q, p_p) , X_{\mathcal{B}} \rangle. \nonumber
    \end{gather}
\end{definition}

To determine the optimal control we need minimize $\Tilde{H}$. We will make the assume that 
\begin{equation}\label{eq:optimal_control}
    \frac{\partial \Tilde{H}}{\partial B} = 0 
\end{equation}
has a unique solution and completely determines the optimal control.

We will call the optimal Hamiltonian $H_{opt}:= \min_{B}\Tilde{H}$. An optimal trajectory is given be an integral curve of Hamilton's equations 
\begin{equation*}
\begin{split}
     \dot{q} &= \frac{\partial H_{opt}}{\partial p_q }, \quad \dot{p}_q = -\frac{\partial H_{opt}}{\partial q },\\
     \dot{p} &= \frac{\partial H_{opt}}{\partial p_p}, \quad \dot{p}_p = - \frac{\partial H_{opt}}{\partial p},
\end{split}
\end{equation*}

subject to the prescribed boundary conditions
\begin{equation*}
    \begin{array}{cc}
    q(0) = q_0, & q(T) = q_f, \\
    p(0) = p_0, & p(T) = p_f.
    \end{array}
\end{equation*}

In order to simplify notation let $(q, p) = x$, $(p_q, p_p) = y$, and the Hamiltonian vector field by $Y = (\dot{x}, \dot{y})$. 
Letting $\Omega = dx \wedge dy$ be the cannonical symplectic form on $T^*\mathcal{D}'$, the optimal control equations can be written as:
\begin{equation*}
    i_Y\Omega = dH_{opt}.
\end{equation*}
\begin{proposition}\label{prop:2_const_of_motion}
The extended optimal control Hamiltonian flow for a magnetically controlled system has at least two constants of motion i.e. there exists $E_1, E_2 : T^*\mathcal{D}' \rightarrow \mathbb{R}$ such that $\mathcal{L}_YE_i = 0, \ i \in \{1, 2\}$.

\end{proposition}

\begin{proof}
$E_1 = H_{opt}$ is conserved by the definition of $Y$. We now claim that $E_2 =   \pi_{\mathcal{D}'}^*H$ is also a conserved quantity, where $\pi_{\mathcal{D}'}:T^*\mathcal{D}' \rightarrow \mathcal{D}'$ is the canonical projection. 
\begin{equation*}
    \begin{split}
        \mathcal{L}_Y(\pi_{\mathcal{D}'}^*H) = d(\pi_{\mathcal{D}'}^*H)(Y) &= \pi_{\mathcal{D}'}^* dH (Y) \\
        &= dH((\pi_{\mathcal{D}'})_*Y).
    \end{split}
\end{equation*}
But $\pi_{\mathcal{D}'}Y = X_{\mathcal{B}_{opt}}$, where $\mathcal{B}_{opt}$ is the optimal magnetic field that satisfies \eqref{eq:optimal_control}. By Proposition \ref{prop:energy_conservation_magnetic} we know 
$$\mathcal{L}_{X_\mathcal{B}}H = dH(X_\mathcal{B}) = 0 $$
for any $B$, so in particular it must vanish for $B_{opt}$. Hence 
$$ \mathcal{L}_Y(\pi_{\mathcal{D}'}^*H) = dH(X_{\mathcal{B}_{opt}}) = 0. $$
\end{proof}

This result has an important consequence for two dimensional systems.
\begin{corollary}\label{cor:integrable}
    Any magnetically controlled optimal control system on a 2-dimensional manifold is completely integrable i.e. if either \eqref{eq:eom_general_hamiltonian} or \eqref{eq:magnetic_Lie-Poisson} can be reduced to a 2-dimensional problem, the system $(T^*{\mathcal{D}'}, H_{opt}, \Omega)$ is completely integrable.
\end{corollary}
\begin{proof}
    Since $E_1 = \pi_{\mathcal{D}}^*H$ and $E_2 = H_{opt}$ are constants of motion, the components  $(T^*{\mathcal{D}'},  \Omega, \{E_1, E_2\})$ form an integrable system \cite{completely_integrable_review}.
\end{proof}

\section{CHAPLYGIN SLEIGH IN A MAGNETIC FIELD}\label{sec:chap}
We will apply the theory  to the case of a magnetically controlled Chaplygin sleigh as shown in Fig. \ref{fig:Dora's cool pic}. Such a system can be thought of as an ice skater \cite{figure_skate}. The state space is given by $Q = G = \mathrm{SE}_2$. The coordinates $(x,y,\theta)$ represent the coordinates of the contact point along with its orientation.
We will assume that the electric charge is concentrated at the center of mass and is acted on by a vertical magnetic field of strength $B$. The coordinates of the center of mass are 
$x_c = x+a\cos\theta$, $y_c = y+a\sin\theta$. The magnetic field can be written as:
\begin{equation*}
\begin{split}
    \mathcal{B} &= B dx_c\wedge dy_c \\
    &= B\left( dx\wedge dy + a\cos\theta dx\wedge d\theta + a\sin\theta dy\wedge d\theta\right).
\end{split}
\end{equation*}



In coordinates $(x,y,\theta)$, the mass matrix for the Chaplygin sleigh is: 
\begin{equation*}
	M = \begin{bmatrix}
		m & 0 & -ma\sin\theta \\
		0 & m & ma\cos\theta \\
		-ma\sin\theta & ma\cos\theta & I+ma^2
	\end{bmatrix},
\end{equation*}
and the nonholonomic constraint is given by
$$\dot{y}\cos\theta - \dot{x}\sin\theta= 0,$$
and is equivalent to prohibiting movement perpendicular to the forward orientation.

Using \eqref{eq:eom_general_hamiltonian} the equations of motion are 
\begin{equation*}
	\begin{split}
		&\ddot{x} - a\dot\theta^2\cos\theta - a\ddot\theta\sin\theta + \frac{\lambda}{m}\sin\theta = \\
		&\hspace{1.25in}\frac{eB}{m}(-\dot{y}-a\dot\theta\cos\theta) \\
		&\ddot{y} - a\dot\theta^2\sin\theta + a\ddot\theta\cos\theta-  \frac{\lambda}{m}\cos\theta  = \\ &\hspace{1.25in}\frac{eB}{m}(\dot{x}-a\dot\theta\sin\theta)\\
		&(I+ma^2)\ddot\theta + ma\dot\theta(\dot{x}\cos\theta + \dot{y}\sin\theta)= \\
		&\hspace{1.25in}eB\Big(a\dot{x}\cos\theta + a\dot{y}\sin\theta\Big)
	\end{split}
\end{equation*}

Let $v = \dot{x}\cos\theta + \dot{y}\sin\theta$ be its forward velocity and let $\omega = \dot{\theta}$ be the angular momentum of the ice skater. We then relate $(x, y, \theta)$ to this new coordinate system $(v, \omega)$ through: $\dot{x} = v\cos\theta$, $\dot{y} = v\sin\theta$, and $\dot\theta = \omega$.
In coordinates $(v,\omega)$, the equations of motion reduce to (which is equivalent to the reduced equation \eqref{eq:magnetic_Lie-Poisson}):
\begin{equation*}
    \begin{split}
        	\dot{\omega} &= \frac{1}{I+ma^2}\left[ - ma\omega v + eBav\right], \\
        	\dot{v} &= a\omega^2 - ea \omega B.
    \end{split}
\end{equation*}

Let us non-dimensionalize the system by rescaling $\tilde{\omega} = \omega/\omega_0$, $\tilde{v} = v/v_0$, $\tilde{B} = B/B_0$, and $\tau = t/t_0$ according to Table \ref{table:change_of_coords}, with
$\Omega_0$ being the cycloton frequency; $\Omega_0 = eB_0/m$.
This produces the dimensionless system
\begin{equation*}
    \begin{split}
        \dot{\tilde{v}} &= -c\tilde{\omega}\left(\tilde{B}-\tilde{\omega}\right), \\
        \dot{\tilde{\omega}} &= \tilde{v}\left(\tilde{B}-\tilde{\omega}\right),
    \end{split}
\end{equation*}
where
\begin{equation*}
    c = \frac{ma^2}{I+ma^2}.
\end{equation*}
This is an affine control system in a two dimensional space. By Proposition \ref{prop:2_const_of_motion}, the optimal control Hamiltonian system will be completely integrable. Henceforth, we will drop the tilde on the variables but they will remain dimensionless.

\begin{center}
    \begin{table}
    \centering
    \vspace{0.1in}
        \begin{tabular}{ |c|c| }
            \hline
            Variable & Normalization Factor\\ \hline
            $t$ & $\Omega_0^{-1}$ \\[1.75ex]
            $B$ & $\dfrac{m\Omega_0}{e}$ \\[1.75ex] 
            $v$ & $\dfrac{(I + ma^2)\Omega_0}{ma}$ \\[1.75ex]
            $\omega$ & $\Omega_0$\\
            \hline
        \end{tabular}
    \vspace{4pt}
    \caption{\label{table:change_of_coords} Table showing the normalisation factors required in order to make the system dimensionless}
    \end{table}
\end{center}

Let the running cost be $\ell = 1/2B^2$. The extended Hamiltonian, \eqref{eq:extende_Hamiltonian}, is
\begin{equation}
	\tilde{H} = \frac{1}{2}B^2 + p_v\left(c\omega^2 - c\omega B \right) +p_\omega\left(-\omega v + v B\right),
\end{equation}
which is optimized to be
\begin{equation*}
	H_{opt} = c\omega^2p_v - \frac{1}{2}p_\omega^2v^2 - \frac{1}{2}c^2\omega^2p_v^2 - \omega p_\omega v + c\omega p_\omega p_v v
\end{equation*}
with equations of motion
\begin{equation*}
	\begin{split}
		\dot{v} &= c\omega^2 -c^2 \omega^2 p_v +  c \omega v p_\omega \\
		\dot{\omega} &= -\omega v + c\omega v p_v -  v^2p_\omega \\
		\dot{p}_v &= \omega p_\omega +  vp_\omega^2 -c \omega p_\omega p_v \\
		\dot{p}_\omega &=  vp_\omega - 2c\omega p_v + c^2\omega p_v^2 -cvp_\omega p_v
	\end{split}
\end{equation*}
Even though we have reduced the dimension of the optimal control problem from $6$ to $4$ by passing to the $(v, \omega)$ coordinates, we can further simplify the dynamics using Corollary \ref{cor:integrable}. As these controls are energy-preserving, the original energy is a constant of motion. In terms of $(v, \omega)$ the energy can be written as:
\begin{equation*}
	E = \frac{1}{2}\left(\frac{1}{c}v^2 + \omega^2\right).
\end{equation*}
We use this fact to perform yet another change of coordinates:
\begin{equation*}
     v = \sqrt{c}\cos \alpha, \ \omega = \sin \alpha.
\end{equation*}
The $\alpha$ equation of motion is given by
\begin{equation*}
    \dot{\alpha} = \sqrt{c}(B - \sin\alpha).
\end{equation*}
The extended Hamiltonian for this reduced case is 
\begin{equation*}
    \tilde{H}(B, \alpha, p_\alpha) = \ell(B) + p_\alpha  \sqrt{c}(B - \sin\alpha).
\end{equation*}
If we use $\ell(B) = 1/2B^2$ as earlier, we find the optimal control
\begin{equation*}
    B^* = \arg \min_{B'} \tilde{H}(B',\alpha, p_\alpha)  = - \sqrt{c} p_\alpha.
\end{equation*}
Plugging this into the Hamiltonian gives up $H_{opt}$ in terms of $p_\alpha$ and $\alpha$. We scale $p_\alpha$ by a factor of $-\sqrt{c}$ in order to obtain a Hamiltonian in terms of  $B$ and $\alpha$ so that we can solve the system of equations of motion \eqref{eq:alphadotoptimal}, \eqref{eq:Bdotoptimal} for the magnetic field directly.
This gives
\begin{equation}
\label{eq:alphaH}
    H_{opt}(\alpha, B) = \frac{1}{2} \frac{B^2}{c} - B \sin \alpha,
\end{equation}
and the equations of motion
\begin{align}
\label{eq:alphadotoptimal}
    \dot{\alpha} &= \frac{1}{c}B - \sin\alpha, \\
\label{eq:Bdotoptimal}
    \dot{B} &= B \cos \alpha.
\end{align}

\section{NUMERICAL RESULTS}\label{sec:results}

\begin{figure}
    \centering
    \includegraphics[width=0.48\textwidth]{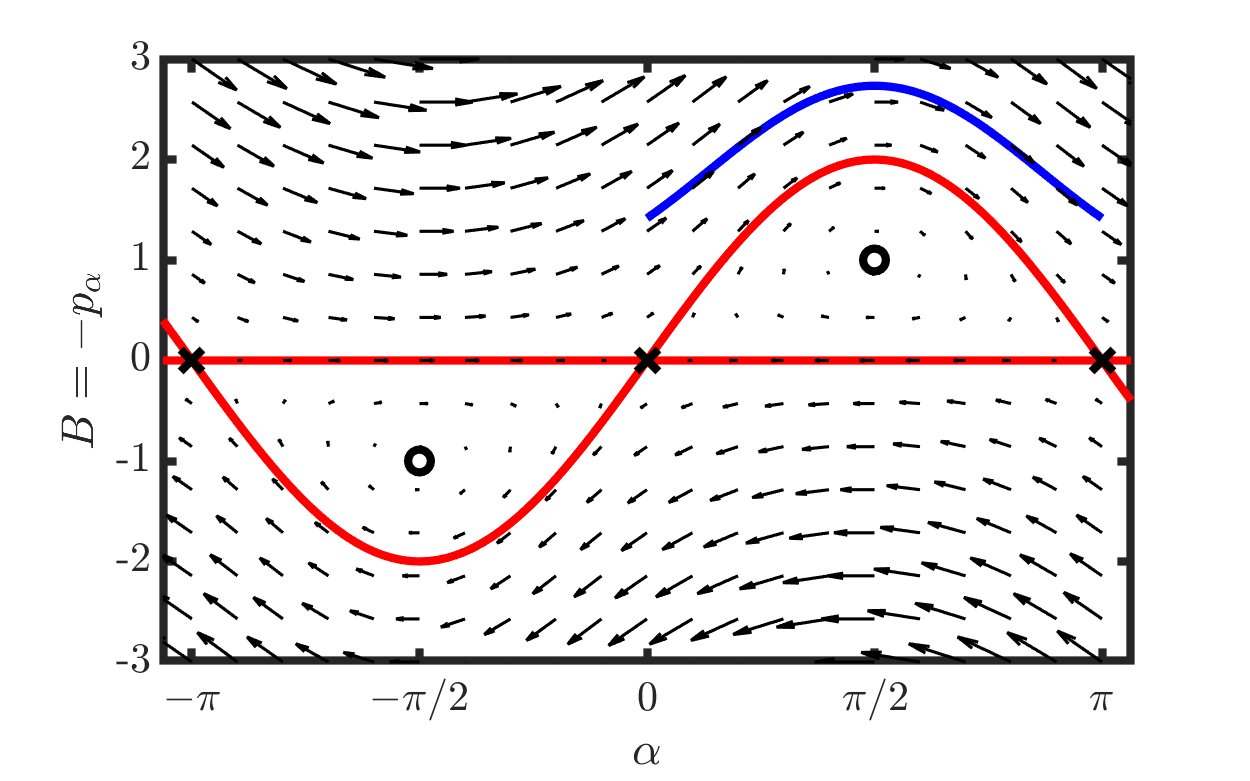}
    \caption{The phase portrait for the controlled Hamiltonian \eqref{eq:alphaH} with $c=1$. The stable equilibria are marked with circles, and the unstable equilibria are marked with `x's. The homoclinic orbits $P = \pm\sin\alpha$ are graphed in red. The optimal trajectory in Fig.~\ref{fig:optimal_trajectories} is marked in blue.}
    \label{fig:phase_portrait}
\end{figure}

We begin by discussing the phase portrait of the optimized Hamiltonian \eqref{eq:alphaH} (see Fig.~\ref{fig:phase_portrait}). There are four critical points of the control for $\alpha \in S^1$ and $B \in \mathbb{R}$. Two stable critical points occur at $(\alpha, B) = (\pm \pi/2, \pm 1)$, corresponding to the Chaplygin sleigh rotating in place with a constant angular velocity $\omega = \pm c^{-1}$. Two unstable points occur at $(0, 0)$ and $(\pi, 0)$, corresponding to the sleigh moving in a straight forwards or backwards respectively. 

There are two heteroclinic orbits connecting the unstable points, defined by the equations $B=0$ and $B=2 \sin \alpha$. The $B=0$ heteroclinic orbit reduces to the uncontrolled system, where the Chaplygin is attracted to going straight, i.e.~$\dot{\alpha} = - \sin \alpha$. The $B=2\sin \alpha$ orbit reverses time in that system, giving $\dot{\alpha} = \sin \alpha$. This means that if we use this orbit, the optimal control orbit effectively makes the backwards direction stable and the forwards direction unstable. 

In order to have controllability in this system, the value of the Hamiltonian \eqref{eq:alphaH} must exceed the Hamiltonian value of the connecting orbits, or
\begin{equation*}
    \frac{ B^2}{2c} - B \sin \alpha > 0.
\end{equation*}
This implies that there is a minimum value of the magnetic field for controllability
\begin{equation}
\label{eq:minB}
    \abs{B} > \max_{\alpha} (2 c\sin \alpha) = 2c.
\end{equation}

Now, we consider the problem 
\begin{gather}\label{eq:fixedtimeproblem}
    J = \min_{B} \, \int_0^T \, \frac{1}{2}B^2\, dt, \\
    \alpha(0) = 0, \quad \alpha(T) = \pi, \nonumber
\end{gather}
corresponding to the problem of turning the Chaplygin sleigh around from a straight forward trajectory in a fixed amount of time $T$. Solutions to this problem are symmetric with $(\alpha, B) \mapsto (-\alpha, -B)$, so we only consider $B > 0$.  In order to do this numerically, we begin by noting that (\ref{eq:alphadotoptimal}, \ref{eq:Bdotoptimal}) can be combined to find the equation for $\alpha$ as
\begin{equation*}
    \ddot{\alpha} = \frac{1}{2} \sin(2\alpha),
\end{equation*}
i.e., the evolution of $\alpha$ in time exactly matches the nonlinear pendulum with frequency $2$ and a $\pi/2$ phase shift. We note that this result holds even for any value of $c$. This means that the time to get from our two end points can be given exactly by the formula
\begin{align}
\label{eq:Telliptic}
    T &= \frac{K(-B(0)^{-2})}{\sqrt{|B(0)|}},
\end{align}
where $K$ is the complete elliptic integral of the first kind. Then, using the fact that $T$ is monotonic in $B(0)$, we can solve \eqref{eq:Telliptic} using a bracketed root-finding technique with a high-accuracy elliptic integral for an initial $B(0)$ near machine precision. If we do not consider solutions that do a single half turn, this gives a unique solution to the Pontryagin optimal control for \eqref{eq:fixedtimeproblem}.

\begin{figure}
    \centering
    \includegraphics[width=0.48\textwidth]{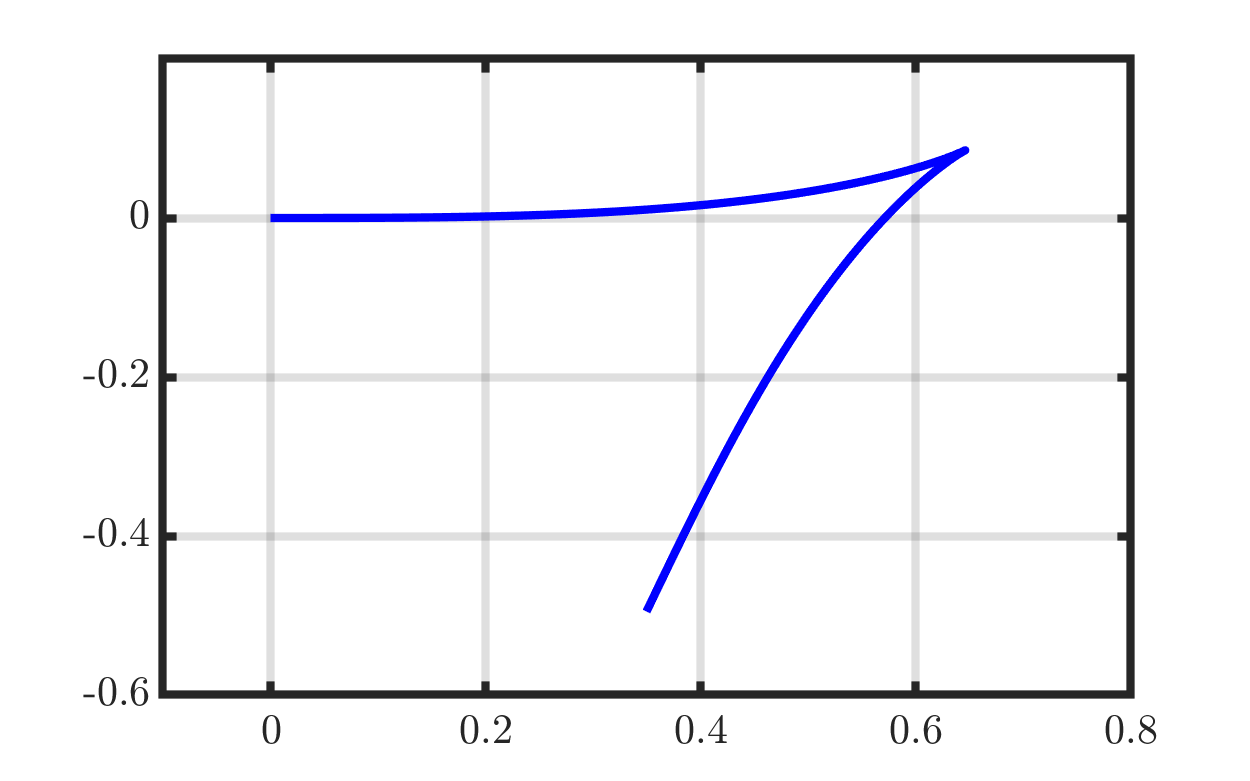}
    \caption{Spatial $(x,y)$ plots of the optimal trajectories of the Chaplygin sleigh. The trajectory in terms of $\alpha$ and $B$ is shown as the blue line Fig.~\ref{fig:phase_portrait}}
    \label{fig:optimal_trajectories}
\end{figure}

In Fig.~\ref{fig:optimal_trajectories}, we plot an example optimal trajectory for \eqref{eq:fixedtimeproblem} in the $(x,y)$-plane. We solve for the initial magnetic field $B(0)$ for the time $T=2$ using MATLAB's \texttt{fzero} and \texttt{ellipticK} functions for root finding and evaluating the elliptic integral respectively. The solution to this is evolved via a Runge-Kutta scheme (\texttt{ode45}) to obtain the spatial path, with initial conditions $(x,y,\theta) = (0,0,0)$.

\begin{figure}
    \centering
    \includegraphics[width=0.48\textwidth]{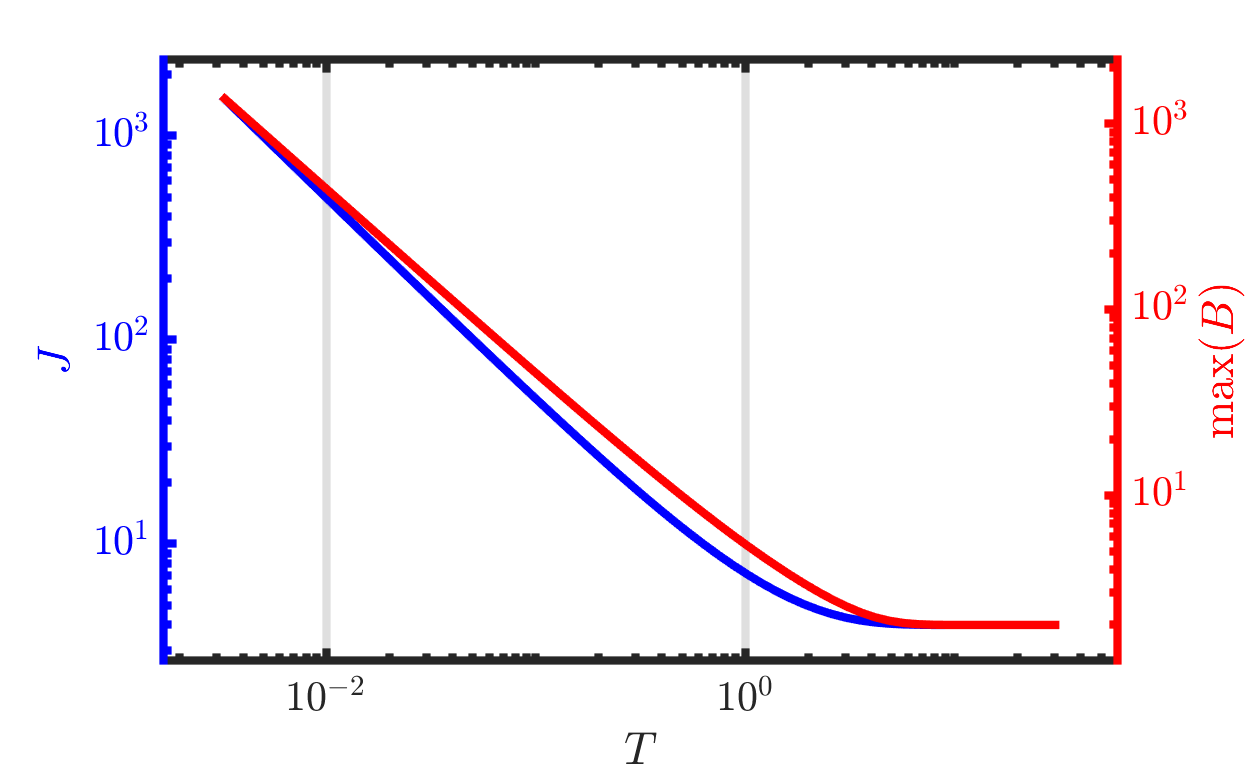}
    \caption{Cost of to turn the sleigh around as a function of final time. }
    \label{fig:cost_vs_time}
\end{figure}

In Fig.~\ref{fig:cost_vs_time}, we plot both the cost and the maximum magnetic field over the trajectories for a range of $T$ values between $T=10^{-1.5}$ through $T=10^{2.5}$. The cost was obtained by integrating over the trajectory using \texttt{ode45}, and the maximum magnetic field was found by noting that the Hamiltonian $B$ is maximized at $\alpha = \pi/2$ and using conservation of \eqref{eq:alphaH} to find
\begin{equation*}
    \max B = 1 + \sqrt{1 + B(0)^2}. 
\end{equation*}
We see that the maximum magnetic field decreases as a function of $T$ and is asymptotic to the minimum value of $2$ from \eqref{eq:minB}. Additionally, we see that the cost is decreasing as $T$ increases. We can find the value it converges to by analytically evaluating the cost of the heteroclinic orbit $B = 2\sin \alpha$:
\begin{align*}
    J_{\min} = \int_{-\infty}^{\infty}\frac{1}{2}B^2 dt = 4.
\end{align*}
\section{CONCLUSIONS}\label{sec:conclusions}

    

We investigated the optimal control problem for nonholonomic Hamiltonian systems subject to a magnetic field. An example of an electrically charged Chaplygin Sleigh was presented. Due to the energy-preserving property of the magnetic controls, the resulting optimal control problem is always completely integrable; in our specific example, the equations of motion were equivalent to the nonlinear pendulum and solutions were found via elliptic integrals.

A direction for future work is on the question of controllability when the magnetic field is underactuated. 
In this work, we assume that we can reach any location in a given time. 
That is, for any pair of points $x,y\in \mathcal{D}'$ and time $T>0$, does there exist a control law that drives the system from $x$ to $y$ while obeying the magnetic nonholonomic equations of motion?
Whether or not this is actually possible is generally nontrivial.

Another possible research direction, specifically for the Chaplygin sleigh, is on trajectory-tracing as discussed in \cite{figure_skate}. As energy is preserved, this places bounds on the maximum angular velocity which will make tracing an arbitrary path impossible. 

A final immediate research direction is to extend this procedure to either relativistic or quantum systems.



\addtolength{\textheight}{-12cm}   





\section*{ACKNOWLEDGMENT}
We thank Mallory Gaspard for insightful conversations and her enthusiasm.
\newpage

\bibliographystyle{ieeetr}
\bibliography{references}

\end{document}